\documentclass[11pt]{amsart}
\usepackage{amsfonts,amssymb,amsthm}
\usepackage{amsmath,amscd}
\usepackage{pstricks}
\usepackage{pstricks,pst-node}
\usepackage{mathrsfs}
\usepackage[all]{xy}

\DeclareMathAlphabet{\mathpzc}{OT1}{pzc}{m}{it}

\renewcommand{\subsection}[1]{\vspace{.18in}
\par\noindent\addtocounter{subsection}{1}
\setcounter{equation}{0}{\bf\thesubsection.\hspace{5pt}#1}}

\theoremstyle{definition}
\newtheorem{Def}[subsection]{Definition}
\newtheorem{Rem}[subsection]{Remark}
\newtheorem{Rems}[subsection]{Remarks}

\theoremstyle{plain}

\newtheorem{Prop}[subsection]{Proposition}
\newtheorem{Thm}[subsection]{Theorem}

\newtheorem{Lem}[subsection]{Lemma}
\newtheorem{Coro}[subsection]{Corollary}
\numberwithin{equation}{subsection}
\newtheorem{Conj}[subsection]{Conjecture}

\newcommand{\fD}{{\mathfrak D}}
\newcommand{\Vn}{\sV(n)}

\newcommand{\Wnz}{\sW_\mbz(n)}
\newcommand{\Wnk}{\sW_k(n)}
\newcommand{\afVn}{\sV_\vtg(n)}
\newcommand{\afbfVn}{\boldsymbol{\sV}_\vtg(n)}
\newcommand{\Vnz}{\sV^0(n)}
\newcommand{\bfVn}{\boldsymbol{\sV}(n)}
\newcommand{\bfVnz}{\boldsymbol{\sV}^0(n)}

\newcommand{\bse}{\boldsymbol{e}}

\newcommand{\bfj}{{\mathbf{j}}}

\newcommand{\bfl}{{\mathbf{0}}}

\newcommand{\bft}{{\mathbf{t}}}

\newcommand{\bfU}{{\mathbf{U}}}

\def\fS{{\frak S}}

\def\fB{{\frak B}}

\newcommand{\msD}{\mathscr D}

\def\sH{{\mathcal H}}

\def\sU{{\mathcal U}}
\def\sV{{\mathcal V}}
\def\sW{{\mathcal W}}
\def\sX{{\mathcal X}}

\def\sZ{{\mathcal Z}}

\newcommand{\mbzn}{\mathbb Z^{n}}
\newcommand{\mbnn}{\mathbb N^{n}}
\newcommand{\mbn}{\mathbb N}
\newcommand{\mbq}{\mathbb Q}

\newcommand{\mbz}{\mathbb Z}

\newcommand{\End}{\operatorname{End}}

\newcommand{\spann}{\operatorname{span}}
\newcommand{\diag}{\operatorname{diag}}

\def\ro{\text{\rm ro}}
\def\co{\text{\rm co}}

\newcommand{\la}{{\lambda}}
\newcommand{\La}{\Lambda}
\newcommand{\ga}{{\gamma}}

\newcommand{\Th}{\Theta}
\newcommand{\dt}{\delta}

\newcommand{\up}{v}

\newcommand{\vep}{\varepsilon}

\newcommand{\al}{\alpha}
\newcommand{\bt}{\beta}
\newcommand{\sg}{\sigma}

\newcommand{\ol}{\overline}

\newcommand{\lan}{\langle}
\newcommand{\ran}{\rangle}
\newcommand{\dleb}{\left[\!\!\left[}
\newcommand{\leb}{\left[}
\newcommand{\bbl}{\big[}
\newcommand{\bbr}{\big]}
\newcommand{\dbbl}{\big[\!\!\big[}
\newcommand{\dbbr}{\big]\!\!\big]}
\newcommand{\drib}{\right]\!\!\right]}
\newcommand{\rib}{\right]}

\def\ddet#1{|\!| #1 |\!|}

\def\ggp#1#2{\left[\kern-3.2pt\left[{#1\atop #2}\right]\kern-3.2pt\right]}

\newcommand{\p}{\prec}
\newcommand{\pr}{\preccurlyeq}
\def\leq{\leqslant}\def\geq{\geqslant}
\def\le{\leqslant}\def\ge{\geqslant}

\newcommand{\bop}{\bigoplus}

\newcommand{\ot}{\otimes}

\newcommand{\han}{\subseteq}

\newcommand{\h}{\widehat}
\newcommand{\ti}{\widetilde}

\newcommand{\Lanr}{\Lambda(n,r)}

\newcommand{\Thn}{\Th(n)}
\newcommand{\Thnpm}{\Th^\pm(n)}
\newcommand{\Thnp}{\Th^+(n)}
\newcommand{\Thnm}{\Th^-(n)}
\newcommand\Thnr{\Theta(n,r)}

\newcommand{\lra}{\longrightarrow}
\newcommand{\ra}{\rightarrow}
\newcommand{\map}{\mapsto}

\newcommand{\bfUn}{{\mathbf U}(n)}
\newcommand{\bfUnz}{{\mathbf U}^0(n)}
\newcommand{\Un}{U(n)}
\newcommand{\Unk}{U_k(n)}

\newcommand{\Unz}{U_\mbz(n)}
\newcommand{\barUnk}{\ol{U_k(n)}}

\newcommand{\vtg}{{\!\vartriangle\!}}

\newcommand{\fSr}{\fS_r}

\newcommand{\Hr}{{\sH(r)}}

\newcommand{\bfHr}{{\boldsymbol{\mathcal H}}(r)}

\newcommand{\afE}{E^\vartriangle}

\newcommand{\Sr}{{\mathcal S}(n,r)}
\newcommand{\Srk}{{\mathcal S}_k(n,r)}
\newcommand{\Srq}{{\mathcal S}_\mbq(n,r)}
\newcommand{\afSr}{{\mathcal S}_{\vtg}(n,r)}

\newcommand{\bfSr}{{\boldsymbol{\mathcal S}}(n,r)}
\newcommand{\afbfSr}{{\boldsymbol{\mathcal S}}_\vtg(n,r)}

\newcommand{\afmbnn}{\mathbb N_\vtg^{n}}
\newcommand{\afmbzn}{\mathbb Z_\vtg^{n}}

\newcommand{\afThn}{\Theta_\vtg(n)}

\newcommand{\afThnpm}{\Theta_\vtg^\pm(n)}

\newcommand{\afThnr}{\Theta_\vtg(n,r)}

\newcommand{\afLa}{\Lambda_\vtg}
\newcommand{\afLanr}{\Lambda_\vtg(n,r)}

\begin{document}
\title{BLM realization for the integral form of quantum $\frak{gl}_n$}

\author{Qiang Fu}
\address{Department of Mathematics, Tongji University, Shanghai, 200092, China.}
\email{q.fu@tongji.edu.cn}


\thanks{Supported by the National Natural Science Foundation
of China, the Program NCET, Fok Ying Tung Education Foundation
 and the Fundamental Research Funds for the Central Universities}

\begin{abstract}
Let $\bfUn$ be the quantum enveloping algebra of ${\frak {gl}}_n$ over $\mbq(v)$, where $v$ is an indeterminate. We will use $q$-Schur algebras to realize the integral form of $\bfUn$. Furthermore we will use this result to realize quantum $\frak{gl}_n$ over $k$, where $k$ is a field containing an $l$-th primitive root $\varepsilon$ of $1$ with $l\geq 1$ odd.
\end{abstract}

 \sloppy \maketitle

\section{Introduction}
It is well known that the positive part  of the integral form of quantum enveloping algebras of finite type was realized as a Ringel--Hall algebra (see \cite{R90,R932}).  Using a beautiful geometric construction of $q$-Schur algebras, the entire quantum $\frak{gl}_n$ over the rational function field $\mbq(\up)$ (with $v$ being an indeterminant) was realized by A. A. Beilinson, G. Lusztig and R. MacPherson (BLM) in \cite{BLM}.

Let $\Un$ be the Lusztig $\sZ$-form of quantum $\frak{gl}_n$, where $\sZ=\mbz[v,v^{-1}]$. We will give BLM realization of $\Un$ in this paper. More precisely, We will construct a certain $\sZ$-submodule of $\prod_{r\geq 0}\bfSr$, denoted by $\Vn$, where $\bfSr$ is the $q$-Schur algebra over $\mbq(v)$. We will show that $\Vn$ is a $\sZ$-subalgebra of $\prod_{r\geq 0}\bfSr$ and prove in \ref{realization} that $\Vn$ is isomorphic to $\Un$ as a $\sZ$-algebra. Similarly, we may construct the affine version of $\Vn$, denoted by $\afVn$, which is a certain $\sZ$-submodule of $\prod_{r\geq 0}\afbfSr$, where $\afbfSr$ is the affine $q$-Schur algebra over $\mbq(v)$. We conjecture that $\afVn$ is a $\sZ$-subalgebra of $\prod_{r\geq 0}\afbfSr$. If this conjecture is true, then $\afVn$ is isomorphic to
the $\sZ$-module $\ti{\fD}_\vtg(n)$ defined in \cite[(3.8.1.1)]{DDF}.

Let $k$ be a field containing
an $l$-th primitive root $\varepsilon$ of $1$ with $l\geq 1$ odd.
Specializing $v$ to $\varepsilon$, $k$ will be viewed as a $\sZ$-module.
Let $\Unk=\Un\ot_\sZ k$ and $\barUnk=\Unk/\lan K_i^l-1\mid 1\leq i\leq n-1\ran.$
We will prove that the algebra $\barUnk$ can be realized as a $k$-subalgebra of
$\prod_{r\geq 0}\Srk$, where $\Srk$ is the $q$-Schur algebra over $k$.

We organize this paper as follows. We recall some results of quantum $\frak{gl}_n$ and $q$-Schur algebras in \S2. We will establish some useful
multiplication formulas for $q$-Schur algebras
in \ref{formula 1} and \ref{formula 2}. A certain $\sZ$-submodule of $\prod_{r\geq 0}\bfSr$, denoted by $\Vn$, will be constructed in \S4. We will use \ref{formula 1} and \ref{formula 2} to prove that $\Vn$ is BLM realization of $\Un$. Furthermore, we will give realization of $\barUnk$ in \ref{realization of barUnk}.

Throughout this paper, let $\sZ=\mbz[v,v^{-1}]$,  where $v$ is an
indeterminate, and let $\mbq(v)$ be the fraction field of $\sZ$.
For $i\in\mbz$ let $[i]=\frac{v^i-v^{-i}}{v-v^{-1}}$ and $[\![i]\!]=\frac{v^{2i}-1}{v^2-1}$.
For integers $N,t$ with $t\geq 0$, let
\begin{equation*}
\leb{N\atop t}\rib=\frac{[N][N-1]\cdots[N-t+1]}{[t]^!}\in\sZ,\quad
\dleb{N\atop t}\drib=\frac{[\![N]\!][\![N-1]\!]\cdots[\![N-t+1]\!]}{[\![t]\!]^!}\in\sZ
\end{equation*}
where $[t]^{!}=[1][2]\cdots[t]$ and $[\![t]\!]^{!}=[\![1]\!][\![2]\!]\cdots[\![t]\!]$.
For $\mu\in\mbzn$ and $\la\in\mbnn$ let $\bbl{\mu\atop\la}\bbr=\bbl{\mu_1\atop\la_1}\bbr\cdots\bbl{\mu_n\atop\la_n}\bbr.$

\section{The quantum $\frak{gl}_n$ and the $q$-Schur algebra}

The below definition of quantum $\frak{gl}_n$ is a slightly
modified version of Jimbo \cite{Ji}; see \cite{GL,Ta}.
\begin{Def}\label{definition of U(infty)}
The quantum enveloping algebra of ${\frak {gl}}_n$ is the $\mbq(v)$-algebra $\bfUn$
presented by generators
$$E_i,\ F_i\quad(1\leq i\leq n-1),\  K_j,\  K_j^{-1}\quad(1\leq j\leq n)$$
and relations

$(a)\ K_{i}K_{j}=K_{j}K_{i},\ K_{i}K_{i}^{-1}=1;$

$(b)\ K_{i}E_j=v^{\dt_{i,j}-\dt_{i,j+1}} E_jK_{i};$

$(c)\ K_{i}F_j=v^{\dt_{i,j+1}-\dt_{i,j}} F_jK_i;$

$(d)\ E_iE_j=E_jE_i,\ F_iF_j=F_jF_i\ when\ |i-j|>1;$

$(e)\ E_iF_j-F_jE_i=\delta_{i,j}\frac {\widetilde
K_{i}-\widetilde K_{i}^{-1}}{v-v^{-1}},\ where \
\widetilde K_i =K_{i}K_{i+1}^{-1};$

$(f)\
E_i^2E_j-(v+v^{-1})E_iE_jE_i+E_jE_i^2=0\
 when\ |i-j|=1;$

$(g)\
F_i^2F_j-(v+v^{-1})F_iF_jF_i+F_jF_i^2=0\
 when\ |i-j|=1.$
\end{Def}

Following \cite{Lu90}, let $\Un$ be the $\sZ$-subalgebra of $\bfUn $
generated by all $E_i^{(m)}$, $F_i^{(m)}$, $K_i^{\pm 1}$ and
$\bbl {K_i;0 \atop t} \bbr$, where for $m,t\in\mathbb N$,
$$E_i^{(m)}=\frac{E_i^m}{[m]^!},\,\,F_i^{(m)}=\frac{F_i^m}{[m]^!},\text{ and }
\bigg[ {K_i;0 \atop t} \bigg] =
\prod_{s=1}^t \frac
{K_iv^{-s+1}-K_i^{-1}v^{s-1}}{v^s-v^{-s}}.$$

Let $\Thn$ be the set of all $n\times n$ matrices over
$\mathbb N$. Let $\Thnpm$ be the set of all $A\in\Thn$ whose diagonal entries are zero. Let $\Thnp$ (resp., $\Thnm$) be the subset of $\Thn$ consisting of those matrices $A$ with $a_{i,j}=0$ for all $i>j$ (resp., $i<j$). For $A\in\Thnpm$, write $A=A^++A^-$ with
$A^+\in\Thnp$ and $A^-\in\Thnm$.
For $A\in\Thnpm$ let
\begin{equation}\label{order}
E^{(A^+)}=\prod_{1\leq i\leq h<j\leq n}E_h^{(a_{i,j})}\ \text{and}
\ F^{(A^-)}=\prod_{1\leq j\leq h<i\leq
n}F_h^{(a_{i,j})}\end{equation}
 The orders in which the products $E^{(A^+)}$ and $F^{(A^-)}$ are
taken are defined as follows. Put
$$M_j=M_j(A^+)=E_{j-1}^{(a_{j-1,j})}(E_{j-2}^{(a_{j-2,j})}E_{j-1}^{(a_{j-2,j})})
\cdots(E_{1}^{(a_{1,j})}E_{2}^{(a_{1,j})}\cdots
E_{j-1}^{(a_{1,j})}).$$
 Similarly, put
$$M_j'=(F_{j-1}^{(a_{j,1})}\cdots
F_{2}^{(a_{j,1})}F_{1}^{(a_{j,1})})
\cdots(F_{j-1}^{(a_{j,j-2})}F_{j-2}^{(a_{j,j-2})})
F_{j-1}^{(a_{j,j-1})}.$$
Then $E^{(A^+)}=M_nM_{n-1}\cdots M_2$ and
$F^{(A^-)}=M_2'M_3'\cdots M_n'$.
According to \cite[4.5]{Lu90} and \cite[7.8]{Lu901} we have the following result.
\begin{Prop}\label{basis for Un}
The set $$\{E^{(A^+)}\prod_{1\leq i\leq n}K_i^{\dt_i}\leb{K_i;0\atop\la_i}\rib F^{(A^-)}\mid A\in\Thnpm,\,\dt,\la\in\mbnn,\,\dt_i\in\{0,1\},\,\forall i\}$$ forms a $\sZ$-basis of $U(n)$.
\end{Prop}

Schur algebras are certain important finite-dimensional algebras. It is used to link representation of  general linear groups and symmetric groups. $q$-Schur algebras are quantum deformation of Schur algebras, which is defined by certain endomorphism algebras arising from Hecke algebras of type $A$. We now follow \cite{DJ89,DJ91} to recall the definition of $q$-Schur algebras as follows.
Let $\fS_r$ be the symmetric group on $r$ letters. The symmetric group $\fSr$ is generated by the set $\{s_i:=(i,i+1)\mid 1\leq i\leq r-1\}$.
The  Hecke algebra $\Hr$ associated with $\fS_r$  is  the $\sZ$-algebra generated by
$T_i$ ($1\leq i\leq r-1$),
with the following relations:
$$(T_i+1)(T_i-q)=0,\;\;
T_iT_{i+1}T_i=T_{i+1}T_iT_{i+1},\;\;T_iT_j=T_jT_i\;(|i-j|>1).
$$
where $q=\up^2$. If $w=s_{i_1}s_{i_2}\cdots s_{i_m}$ is reduced let $T_w=T_{i_1}T_{i_2}\cdots T_{i_m}$. Then the set $\{T_w\mid w\in\fS_r\}$ forms a $\sZ$-basis for $\Hr$.
Let $\Lanr=\{\la\in\mbnn\mid\sg(\la):=\sum_{1\leq i\leq n}\la_i=r\}$.
 For $\la\in\La(n,r)$, let $\fS_{\la}$ be the
Young subgroup of $\fS_r$ and let $x_{\la}=\sum_{w\in\frak
S_{\la}}T_w$. Let $\bfHr=\Hr\ot_\sZ\mbq(\up)$. The endomorphism algebras
$$\Sr:=\End_{\Hr}\bigg(
\bop_{\la\in\La(n,r)}x_{\la}\Hr\bigg),\quad \bfSr:=\End_{\bfHr}\bigg(
\bop_{\la\in\La(n,r)}x_{\la}\bfHr\bigg)$$ are called
$q$-Schur algebras.
For $\la,\mu\in\La(n,r)$ let $\msD_{\la,\mu}$ be the set of
distinguished double $(\frak S_\la,\frak S_\mu)$-coset
representatives.  For
$\la,\mu\in\La(\eta,r)$, $d\in\msD_{\la,\mu}$, define
$\phi_{\la\mu}^d\in\Sr$ by
\begin{equation*}
\phi_{\la\mu}^d(x_\nu h)=\delta_{\mu,\nu}\sum_{x\in\fS_\la
d\fS_\mu}T_xh.
\end{equation*} According to \cite[1.4]{DJ91}, the set $\{\phi_{\la\mu}^d\mid \la,\mu\in\Lanr,\,d\in\msD_{\la,\mu}\}$ forms a $\sZ$-basis for $\Sr$.

Let
$\Thnr=\big\{A\in\Thn\,\big|\,\sg(A):=\sum_{1\leq i,j\leq n}a_{i,j}=r\big\}.$
The basis for $\Sr$ can also be indexed by the set $\Thnr$, which we now describe.
For $1\leq i\leq n$, and $\la\in\La(n,r)$ let
\begin{equation*}
R_{i}^{\la}=\bigg\{\sum_{1\leq t\leq i-1}\la_t+1,\sum_{1\leq t\leq i-1}\la_t+2,\ldots,\sum_{1\leq t\leq i-1}\la_t+\la_i\bigg\},
\end{equation*}
According to \cite[1.3.10]{JK}, there is
a bijective map
\begin{equation*}
\jmath:\{(\la, d,\mu)\mid
d\in\msD_{\la,\mu},\la,\mu\in\Lanr\}\lra\Thnr
\end{equation*}
sending $(\la, d,\mu)$ to $A=(a_{k,l})$, where $a_{k,l}=|R_k^\la\cap dR_l^\mu|$ for all $k,l\in\mbz$.
If $\la,\mu\in\Lanr$ and $d\in\msD_{\la,\mu}$ are such that $A=
\jmath(\la,d,\mu)$, let
\begin{equation*}
[A]=\up^{-d_A}\phi_{\la,\mu}^d,\quad\text{ where } \quad
d_{A}=\sum_{1\leq i\leq n\atop i\geq k,j<l}a_{i,j}a_{k,l}.
\end{equation*}
Then the set $\{[A]\mid A\in\Thnr\}$ forms a $\sZ$-basis for $\Sr$.

The geometric definition of $q$-Schur algebra was given in \cite[1.2]{BLM}.
It is proved in \cite[A.1]{Du92} that the two definitions of $q$-Schur algebras are equivalent.
According to  \cite[1.2,1.3]{BLM}, for $\lambda\in\Lanr$ and $A\in\Thnr$, we have
\begin{equation}\label{[diag(la)][A]}
\begin{aligned}
\ [\diag(\la)][A]=
\begin{cases}[A] & \text{if}\ \lambda=ro(A)\\
0 & \text{otherwise;}
\end{cases}
\end{aligned} \quad \text{and}\
\begin{aligned}
\ [A][\diag(\la)]=
\begin{cases}[A] & \text{if}\ \lambda=co(A)\\
0 & \text{otherwise,}
\end{cases}
\end{aligned}
\end{equation}
where $\ro(A)=(\sum_ja_{1,j},\cdots,\sum_ja_{n,j})$ and
$\co(A)=(\sum_ia_{i,1},\cdots,\sum_ia_{i,n})$ are the sequences of
row and column sums of $A$.

The algebra $\bfUn$ and the $q$-Schur algebra $\bfSr$ are
related by an algebra epimorphism $\zeta_r$ which we now
describe. For $A\in\Thnpm$, $\dt\in\mbzn$ and $\la\in\mbnn$
let
\begin{equation*}
\begin{split}
A(\dt,\la,r)&=\sum_{\mu\in\La(n,r-\sg(A))}v^{\mu\centerdot\dt}
\leb{\mu\atop\la}\rib[A+\diag(\mu)]\in\bfSr;\\
A(\dt,r)&=\sum_{\mu\in\La(n,r-\sg(A))}v^{\mu\centerdot\dt}
[A+\diag(\mu)]\in\bfSr,
\end{split}
\end{equation*}
where $\mu\centerdot\dt=\sum_{1\leq i\leq n}\mu_i\dt_i$.
Furthermore we set
\begin{equation*}
\begin{split}
A(\dt,\la)&=(A(\dt,\la,r))_{r\geq 0}\in\prod_{r\geq 0}\bfSr;\\
A(\dt)&=(A(\dt,r))_{r\geq 0}\in\prod_{r\geq 0}\bfSr.
\end{split}
\end{equation*}
Then by definition we have $A(\dt)=A(\dt,\bfl)$, where
$\bfl=(0,\cdots,0)\in\mbnn$.
For $1\leq
i,j\leq n$, let $E_{i,j}\in\Thn$ be the matrix $(a_{k,l})$ with
$a_{k,l}=\delta_{i,k}\delta_{j,l}$.
According to \cite{BLM}, there is
an algebra epimorphism $$\zeta_r:\bfUn
\twoheadrightarrow\bfSr$$ satisfying
$\zeta_r(E_h)=E_{h,h+1}(\mathbf 0,r)$, $\zeta_r(K_1^{j_1}K_2^{j_2}\cdots
K_n^{j_n})=0(\mathbf j,r)$ and $\zeta_r(F_h)=E_{h+1,h}(\mathbf
0,r)$, for $1\leq h\leq n-1$ and $\bfj\in\mbzn$.

We conclude this section by recalling an important triangular relation in $q$-Schur algebras.
For $A=(a_{s,t})\in\Thn$ and $i<j$, let
$\sg_{i,j}(A)=\sum_{s\leq i;t\geq j}a_{s,t}$ and
$\sg_{j,i}(A)=\sum_{s\leq i;t\geq j}a_{t,s}$. Define $A'\pr A$ iff
$\sg_{i,j}(A')\leq\sg_{i,j}(A)$ and $\sg_{j,i}(A')\leq\sg_{j,i}(A)$
for all $1\leq i<j\leq n$. Put $A'\p A$ if $A'\pr A$ and, for some
pair $(i,j)$ with $i\not=j$, $\sg_{i,j}(A')<\sg_{i,j}(A)$.
According to \cite[5.3 and 5.4(c)]{BLM}, we have the following result.
\begin{Prop}\label{tri}
For $A\in\Thnpm$, we have
$$\prod_{1\leq i\leq h<j\leq n}(a_{i,j}E_{h,h+1})(\bfl)
\cdot\prod_{1\leq j\leq h<i\leq
 n}(a_{i,j}E_{h+1,h})(\bfl)=A(\bfl)+f$$
where the ordering of the products in the left hand side of the above equation is the same as in \eqref{order} and
$f$ is the $\mbq(v)$-linear combination of $B(\bfj)$ with $B\in\Thnpm$, $B\p A$ and $\bfj\in\mbzn$.
\end{Prop}

\section{The multiplication formulas for $q$-Schur algebras}

We will derive certain useful multiplication formulas for $q$-Schur algebras in \ref{formula 1} and \ref{formula 2}.

We need  some preparation before proving \ref{formula 1} and \ref{formula 2}.
Let $\bar\ :\sZ\ra\sZ$ be the ring homomorphism defined by $\bar v=v^{-1}$.
The following impotent multiplication formulas for  $q$-Schur algebras was proved in \cite[3.4]{BLM}.
\begin{Prop}\label{BLM formulas}
Let $1\leq h\leq n-1$, $A\in\Thnr$ and $\la=\ro(A)$.   Let $B_m=\diag(\la)+m E_{h,h+1}-m E_{h+1,h+1}$ and
$C_m=\diag(\la)-m E_{h,h}+m E_{h+1,h}.$ Then in $\Sr$

$(1)\ \displaystyle
[B_m]\cdot[A]=\sum_{\bft\in\La(n,m)\atop \forall
u\in\mbz,t_u\leq
a_{h+1,u}}\up^{\bt(\bft,A)}\prod_{u\in\mbz}\ol{\dleb{a_{h,u}+t_u\atop
t_u}\drib} \biggl[
A+\sum_{u\in\mbz}t_u(\afE_{h,u}-\afE_{h+1,u})\biggr];$\\
for all $0\le m\le\la_{h+1}$, where $\bt(\bft,A)=\sum_{j\geq
u}a_{h,j}t_u-\sum_{j>u}a_{h+1,j}t_u+\sum_{u<u'}t_ut_{u'}$.

$(2)\ \displaystyle
[C_m]\cdot[A]=\sum_{\bft\in\La(n,m)\atop \forall
u\in\mbz,t_u\leq
a_{h,u}}\up^{\ga(\bft,A)}\prod_{u\in\mbz}\ol{\dleb{a_{h+1,u}+t_u\atop
t_u}\drib} \biggl[
A-\sum_{u\in\mbz}t_u(\afE_{h,u}-\afE_{h+1,u})\biggr],$\\
for all $0\le m\le\la_{h}$, where $\ga(\bft,A)=\sum_{j\leq
u}a_{h+1,j}t_u-\sum_{j<u}a_{h,j}t_u+\sum_{u<u'}t_ut_{u'}$.
\end{Prop}


We also need the following formulas for Gaussian binomial coefficient (see \cite{Liu}).
\begin{Lem}\label{binomial}
For $m,n\in\mbz$, $a,b\in\mbn$ we have

$(1)$ $\dbbl{n\atop a}\dbbr=\sum\limits_{0\leq j\leq a}v^{2(m-j)(a-j)}
\dbbl{m\atop j}\dbbr\dbbl{n-m\atop a-j}\dbbr;$

$(2)$
$\dbbl{m\atop a}\dbbr \dbbl{m\atop b}\dbbr=\sum\limits_{0\leq c\leq \min\{a,b\}}
v^{2(b-c)(a-c)}\dbbl{m\atop a+b-c}\dbbr \dbbl{a+b-c\atop
c,\,a-c,\,b-c}\dbbr,$
where $\dbbl{a+b-c\atop
c,\,a-c,\,b-c}\dbbr=\frac{[\![a+b-c]\!]^!}{[\![c]\!]^!
[\![a-c]\!]^![\![b-c]\!]^!}.$
\end{Lem}

Let $\leq$ be the partial order on $\mbnn$ defined by setting, for $\la,\mu\in\mbnn$,
$\la\leq\mu$ if and only if $\la_i\leq\mu_i$ for $1\leq i\leq n$.
For $\la,\al,\bt,\ga\in\mbnn$ with $\la=\al+\bt+\ga$ let $$\leb{\la\atop\al,\,\bt,\,\ga}
\rib=\prod\limits_{1\leq i\leq n}\frac{[\la_i]^!}{[\al_i]^![\bt_i]^![\ga_i]^!}.$$
The above lemma immediately yields the following corollary.




\begin{Coro}\label{mul gauss binomial}
For $\la,\mu\in\mbnn$ and $\al,\bt\in\mbzn$ we have

$(1)$
$\bbl{\al+\bt\atop\la}\bbr=\sum\limits_{\mu\in\mbnn,\,\mu\leq\la}
v^{\al\centerdot(\la-\mu)-\mu\centerdot\bt}
\bbl{\al\atop\mu}\bbr\bbl{\bt\atop\la-\mu}\bbr;$

$(2)$
$\bbl{\al\atop\la}\bbr\bbl{\al\atop\mu}\bbr=
\sum\limits_{\ga\in\mbnn\atop\ga\leq\la,\,\ga\leq\mu}v^{\la\centerdot\mu
-\al\centerdot\ga}\bbl{\la+\mu-\ga\atop\ga,\,\la-\ga,\,\mu-\ga}
\bbr\bbl{\al\atop\la+\mu-\ga}\bbr$.
\end{Coro}

We now use \ref{BLM formulas} and \ref{mul gauss binomial} to prove \ref{formula 1} and \ref{formula 2}.

\begin{Lem}\label{formula 1}
For $A\in\Thnpm$, $\la,\mu\in\mbnn$ and $\dt,\ga\in\mbzn$
we have
$$0(\ga,\mu)A(\dt,\la)=\sum_{\nu\in\mbnn,\,\nu\leq\mu}a_\nu A(\ga+\dt-\nu,\la+\mu-\nu),$$
where $0$ stands for the zero matrix and
$$a_\nu=\sum_{\bfj\in\mbnn\atop\nu-\la\leq\bfj\leq\nu}
v^{\ro(A)\centerdot(\ga+\mu-\bfj)+\la\centerdot(\mu-\bfj)}
\leb{\ro(A)\atop\bfj}\rib\leb{\la+\mu-\nu\atop\nu-\bfj,\,\la-\nu+\bfj,\,
\mu-\nu}\rib.$$
\end{Lem}
\begin{proof}
According to \eqref{[diag(la)][A]} we have
\begin{equation*}
\begin{split}
0(\ga,\mu,r)A(\dt,\la,r)&=\sum_{\al\in\La(n,r-\sg(A))}v^{(\ro(A)+\al)
\centerdot\ga+\al\centerdot\dt}\leb{\ro(A)+\al\atop\mu}\rib
\leb{\al\atop\la}\rib[A+\diag(\al)].
\end{split}
\end{equation*}
Furthermore by \ref{mul gauss binomial} we have
\begin{equation*}
\begin{split}
\leb{\ro(A)+\al\atop\mu}\rib
\leb{\al\atop\la}\rib&=
\sum_{\bfj\in\mbnn,\,
\bfj\leq\mu}v^{\ro(A)\centerdot(\mu-\bfj)-\al\centerdot\bfj}
\leb{\ro(A)\atop\bfj}\rib
\leb{\al\atop\mu-\bfj}\rib\leb{\al\atop\la}\rib\\
&=\sum_{\bfj,\bt\in\mbnn,\,
\bfj\leq\mu\atop\bt\leq\la,\,\bt\leq\mu-\bfj}
v^{(\ro(A)+\la)\centerdot(\mu-\bfj)-\al\centerdot(\bfj+\bt)}
\leb{\ro(A)\atop\bfj}\rib
\leb{\al\atop\la+\mu-\bfj-\bt}\rib
\\&\qquad\qquad\times\leb{\la+\mu-\bfj-\bt\atop
\bt,\,\la-\bt,\,\mu-\bfj-\bt}\rib.
\end{split}
\end{equation*}
Thus we conclude that
\begin{equation*}
\begin{split}
0(\ga,\mu,r)A(\dt,\la,r)&=\sum_{\bfj,\bt\in\mbnn,\,
\bfj\leq\mu\atop\bt\leq\la,\,\bt\leq\mu-\bfj}
v^{\ro(A)\centerdot(\ga+\mu-\bfj)+\la\centerdot(\mu-\bfj)}
\leb{\ro(A)\atop\bfj}\rib
\leb{\la+\mu-\bfj-\bt\atop
\bt,\,\la-\bt,\,\mu-\bfj-\bt}\rib\\
&\qquad\qquad\qquad\times A(\ga+\dt-\bfj-\bt,\la+\mu-\bfj-\bt,r)\\
&=\sum_{\nu\in\mbnn,\,\nu\leq\mu}a_\nu A(\ga+\dt-\nu,\la+\mu-\nu,r).
\end{split}
\end{equation*}
The assertion follows.
\end{proof}

For simplicity, we set
$A(\dt,\la,r)= 0$ and $A(\dt,\la)=0$ if $a_{i,j}< 0$ for some $i\not=j$  for $A\in M_n(\mbz)$.

\begin{Lem}\label{formula 2}
Let $A\in\Thnpm$, $\dt\in\mbzn$, $\la\in\mbnn$, $m\in\mbn$ and $1\leq h\leq n-1$.

$(1)$ For $\bft\in\La(n,m)$, $0\leq j\leq\la_h$, $0\leq k\leq\la_{h+1}$, and $0\leq c\leq\min
\{t_h,j\}$, we set
\begin{equation*}
\begin{split}
\al^\bft_{j,c,k}&=
\bigg(\sum_{h>u}t_u+\la_h-j-c\bigg)\bse_h+\bigg(\la_{h+1}-k-\sum_{h+1>u}t_u
\bigg)
\bse_{h+1},\\
\bt^\bft_{j,c,k}&=(t_h+j-c-\la_h)\bse_h+(k-\la_{h+1})\bse_{h+1}
\end{split}
\end{equation*}
and $$f^\bft_{j,c,k}=v^{g^\bft_{j,k}}\prod_{u\not=h}\ol{\dleb{a_{h,u}+t_u\atop t_u}\drib}
\leb{-t_h\atop\la_h-j}\rib\leb{t_h+j-c\atop c,\,t_h-c,\,j-c}\rib
\leb{t_{h+1}\atop\la_{h+1}-k}\rib$$
where $g^\bft_{j,k}=\sum_{j\geq u,\,j\not=h}a_{h,j}t_u-\sum_{j>u,\,
j\not=h+1}a_{h+1,j}t_u+\sum_{u'\not=h,h+1,\,u<u'}t_ut_{u'}-t_h\dt_h
+t_{h+1}\dt_{h+1}+2jt_h-kt_{h+1}$.
Then we have
\begin{equation*}
\begin{split}
&\qquad(mE_{h,h+1})(\bfl)A(\dt,\la)\\
&=\sum_{\tiny{\begin{array}{c}\bft\in\La(n,m)\\
0\leq j\leq\la_h,\,0\leq k\leq\la_{h+1}\\
0\leq c\leq\min
\{t_h,j\}\end{array}}}
f^\bft_{j,c,k}
\bigg(A+\sum_{u\not=h}t_uE_{h,u}-\sum_{u\not=h+1}t_uE_{h+1,u}\bigg)
(\dt+\al^\bft_{j,c,k},\la+\bt^\bft_{j,c,k}).
\end{split}
\end{equation*}

$(2)$ For $\bft\in\La(n,m)$, $0\leq j\leq\la_{h+1}$, $0\leq k\leq\la_{h}$, and $0\leq c\leq\min
\{t_{h+1},j\}$, we set

\begin{equation*}
\begin{split}
\ti\al^\bft_{j,c,k}&=
\bigg(\sum_{h+1<u}t_u+\la_{h+1}-j-c\bigg)
\bse_{h+1}+\bigg(\la_{h}-k-\sum_{h<u}t_u
\bigg)
\bse_{h},\\
\ti\bt^\bft_{j,c,k}&=(t_{h+1}+j-c-\la_{h+1})
\bse_{h+1}+(k-\la_{h})\bse_{h}
\end{split}
\end{equation*}
and
$$\ti f^\bft_{j,c,k}=v^{\ti g^\bft_{j,k}}
\prod_{u\not=h+1}\ol{\dleb{a_{h+1,u}+t_u\atop t_u}\drib}
\leb{-t_{h+1}\atop\la_{h+1}-j}\rib\leb{t_{h+1}+j-c\atop c,\,t_{h+1}-c,\,j-c}\rib
\leb{t_{h}\atop\la_{h}-k}\rib$$
where $\ti g^\bft_{j,k}=\sum_{j\leq u,\,j\not=h+1}a_{h+1,j}t_u-\sum_{j<u,\,
j\not=h}a_{h,j}t_u+\sum_{u\not=h,h+1,\,u<u'}t_ut_{u'}+t_h\dt_h
-t_{h+1}\dt_{h+1}+2jt_{h+1}-kt_{h}$.
Then we have
\begin{equation*}
\begin{split}
&\qquad(mE_{h+1,h})(\bfl)A(\dt,\la)\\
&=\sum_{\tiny{\begin{array}{c}\bft\in\La(n,m)\\
 0\leq j\leq\la_{h+1},\,0\leq k\leq\la_{h}\\
0\leq c\leq\min
\{t_{h+1},j\}\end{array}}}
\ti f^\bft_{j,c,k}
\bigg(A-\sum_{u\not=h}t_uE_{h,u}+\sum_{u\not=h+1}t_uE_{h+1,u}\bigg)
(\dt+\ti\al^\bft_{j,c,k},\la+\ti\bt^\bft_{j,c,k}).
\end{split}
\end{equation*}
\end{Lem}
\begin{proof}
For simplicity, for $A\in M_n(\mbz)$  with $\sg(A)=r$, we set $[A] = 0\in\bfSr$ if $a_{i,j}<0$  for some $i,j$.
According to \ref{BLM formulas} we have
\begin{equation*}
\begin{split}
&\qquad(mE_{h,h+1})(\bfl,r)A(\dt,\la,r)\\
&=\sum_{\al\in\La(n,r-\sg(A))}
v^{\al\centerdot\dt}\leb{\al\atop\la}\rib[mE_{h,h+1}+\diag(\ro(A)+\al
-m\bse_{h+1})]\cdot[A+\diag(\al)]\\
&=\sum_{\al\in\La(n,r-\sg(A))\atop\bft\in\La(n,m)}
v^{\bt(\bft,A+\diag(\al))+\al\centerdot\dt-t_h\al_h}
\prod_{u\not=h}\ol{\dleb{a_{h,u}+t_u\atop t_u}\drib}
\leb{\al_{h}+t_h\atop t_h}\rib\leb{\al\atop\la}\rib\\
&\qquad\qquad\times\bigg[A+\sum_{u\not=h}t_uE_{h,u}
-\sum_{u\not=h+1}t_uE_{h+1,u}+
\diag(\al+t_h\bse_h-t_{h+1}\bse_{h+1})\bigg].
\end{split}
\end{equation*}
Let $\nu=\al+t_h\bse_h-t_{h+1}\bse_{h+1}$. By \ref{mul gauss binomial} we have
\begin{equation*}
\begin{split}
\leb{\nu_h-t_h\atop\la_h}\rib
\leb{\nu_h\atop t_h}\rib&=\sum_{0\leq j\leq\la_h}v^{\nu_h(\la_h-j)+jt_h}\leb{-t_h\atop\la_h-j}\rib
\bigg(\leb{\nu_h\atop t_h}\rib\leb{\nu_h\atop j}\rib\bigg)\\
&=\sum_{0\leq j\leq\la_h\atop 0\leq c\leq\min\{t_h,j\}}
v^{\nu_h(\la_h-j-c)+2jt_h}\leb{-t_h\atop\la_h-j}\rib\leb
{\nu_h\atop t_h+j-c}\rib\leb{t_h+j-c\atop c,\,t_h-c,\,j-c}\rib
\end{split}
\end{equation*}
and
$\bbl{\nu_{h+1}+t_{h+1}\atop\la_{h+1}}\bbr
=\sum_{0\leq k\leq\la_{h+1}}v^{\nu_{h+1}(\la_{h+1}-k)-kt_{h+1}}
\bbl{\nu_{h+1}\atop k}\bbr\bbl{t_{h+1}\atop\la_{h+1}-k}\bbr.$
This implies that
\begin{equation*}
\begin{split}
\leb{\al_h+t_h\atop t_h}\rib\leb{\al\atop\la}\rib&=
\prod_{s\not=h,h+1}\leb{\nu_s\atop\la_s}\rib\sum_{0\leq k\leq\la_{h+1},\,0\leq j\leq\la_h\atop 0\leq c\leq\min\{t_h,j\}}
v^{x^{\nu,\bft}_{j,c,k}}\leb
{-t_h\atop\la_h-j}\rib\leb{t_h+j-c\atop c,\,t_h-c,\,j-c}\rib\\
&\qquad\qquad\times
\leb{t_{h+1}\atop\la_{h+1}-k}\rib\leb{\nu_h\atop t_h+j-c}\rib
\leb{\nu_{h+1}\atop k}\rib.
\end{split}
\end{equation*}
where $x^{\nu,\bft}_{j,c,k}
=\nu_h(\la_h-j-c)+\nu_{h+1}(\la_{h+1}-k)+2jt_h-kt_{h+1}$.
Thus
\begin{equation*}
\begin{split}
&\qquad(mE_{h,h+1})(\bfl,r)A(\dt,\la,r)\\
&=\sum_{\tiny{\begin{array}{c}\bft\in\La(n,m)\\
0\leq j\leq\la_h,\,0\leq k\leq\la_{h+1}\\
0\leq c\leq\min
\{t_h,j\}\end{array}}}\prod_{u\not=h}\ol{\dleb{a_{h,u}+t_u
\atop t_u}\drib}
\leb{-t_h\atop\la_h-j}\rib\leb{t_h+j-c\atop c,\,t_h-c,\,j-c}\rib
\leb{t_{h+1}\atop\la_{h+1}-k}\rib\\
&\qquad\qquad\times \sum_{\nu\in\La(n,r-\sg(A)+t_h-t_{h+1})}v^{y^\nu_{j,c,k}}
\prod_{s\not=h,h+1}\leb{\nu_s\atop\la_s}\rib \leb{\nu_h\atop
t_h+j-c}\rib\leb{\nu_{h+1}\atop k}\rib\\
&\qquad\qquad\qquad\times
\bigg[A+\sum_{u\not=h}t_uE_{h,u}
-\sum_{u\not=h+1}t_uE_{h+1,u}+
\diag(\nu)\bigg]\\
&=\sum_{\tiny{\begin{array}{c}\bft\in\La(n,m)\\
0\leq j\leq\la_h,\,0\leq k\leq\la_{h+1}\\
0\leq c\leq\min
\{t_h,j\}\end{array}}}
f^\bft_{j,c,k}
\bigg(A+\sum_{u\not=h}t_uE_{h,u}-\sum_{u\not=h+1}t_uE_{h+1,u}\bigg)
(\dt+\al^\bft_{j,c,k},\la+\bt^\bft_{j,c,k}).
\end{split}
\end{equation*}
where
$y^{\nu,\bft}_{j,c,k}=\bt(\bft,A+\diag(\al))+\al\centerdot\dt-t_h\al_h
+x^{\nu,\bft}_{j,c,k}=g^\bft_{j,k}+\nu\centerdot(\dt+\al^\bft_{j,c,k})$.
The assertion (1) follows. The assertion (2) can be proved in a way  similar to the proof of (1).
\end{proof}

\section{Realization of $\Un$ and $\barUnk$}

We shall denote by $\Vn$  the $\sZ$-submodule of $\prod_{r\geq 0}\bfSr$ spanned by
$\{A(\dt,\la)\mid A\in\Thnpm,\,\dt\in\mbzn,\,\la\in\mbnn\}.$
Let $\Vnz$ be the $\sZ$-subalgebra of $\prod_{r\geq 0}\bfSr$ generated by $0(\pm\bse_i)$ and $0(0,t\bse_i)$ for $1\leq i\leq n$ and $t\in\mbn$,
where
$\bse_i=(0,\cdots,0,\underset i1,0\cdots,0)\in\mbn^n.$
\begin{Lem}\label{basis of zero part}
The set $\{0(\dt,\la)\mid
\dt,\la\in\mbnn,\,\dt_i\in\{0,1\},\forall i\}$
forms a $\sZ$-basis for $\Vnz$.
\end{Lem}
\begin{proof}
Let $\bfVnz$ be the $\mbq(v)$-subalgebra of $\prod_{r\geq 0}\bfSr$ generated by $0(\pm\bse_i)$ for $1\leq i\leq n$. Since the set $\{0(\bfj)\mid\bfj\in\mbzn\}$ forms a $\mbq(v)$-basis for $\bfVnz$ we conclude that
$\bfVnz$ is isomorphic to $\bfUnz$, where $\bfUnz$ is the $\mbq(v)$-subalgebra of $\bfUn$ generated by $K_i^{\pm 1}$ for $1\leq i\leq n$.
Now the assertion follows from \cite[4.5]{Lu90}.
\end{proof}
We now describe several $\sZ$-bases for $\Vn$ as follows.
\begin{Lem}\label{basis for Vn}
Each of the following set forms a $\sZ$-basis for $\Vn:$

$(1)$ $\frak B_1=\{0(\dt,\la)A(\bfl)\mid A\in\Thnpm,\,
\dt,\la\in\mbnn,\,\dt_i\in\{0,1\},\forall i\};$

$(2)$ $\frak B_2=\{A(\bfl)0(\dt,\la)\mid A\in\Thnpm,\,
\dt,\la\in\mbnn,\,\dt_i\in\{0,1\},\forall i\};$

$(3)$ $\frak B_3=\{A(\dt,\la)\mid A\in\Thnpm,\,
\dt,\la\in\mbnn,\,\dt_i\in\{0,1\},\forall i\}.$
\end{Lem}
\begin{proof}
According to \ref{formula 1} we have
$$0(\dt,\la)A(\bfl)=v^{\ro(A)\centerdot(\dt+\la)}A(\dt,\la)+
\sum_{\bfj\in\mbnn,\,\bfl<\bfj\leq\la}
v^{\ro(A)\centerdot(\dt+\la-\bfj)}\leb{\ro(A)\atop\bfj}\rib
A(\dt-\bfj,\la-\bfj).$$
It follows that $\Vn$ is spanned by
$\{0(\dt,\la)A(\bfl)\mid A\in\Thnpm,\,
\dt\in\mbzn,\,\la\in\mbnn\}$.
Thus by \ref{basis of zero part} we have
$\Vn=\spann_\sZ\fB_1$.
Since the set $\{0(\bfj)A(\bfl)\mid A\in\Thnpm,\,\bfj\in\mbzn\}$ is linearly independent, by \ref{basis of zero part} we conclude that
the set $\fB_1$ is linearly independent. Hence the set $\fB_1$ forms a $\sZ$-basis for $\Vn$. Similarly, the set $\fB_2$ forms a $\sZ$-basis for $\Vn$. It remains to prove that the set
$\fB_3$ forms a $\sZ$-basis for $\Vn$. For $\la\in\mbnn$ and $\mu,\dt\in\mbzn$
we have
$$v^{\dt_i\mu_i}\leb{\mu_i\atop\la_i}\rib
=v^{\la_i}(v^{\la_i+1}-v^{-\la_i-1})v^{(\dt_i-1)\mu_i}\leb{\mu_i\atop
\la_i+1}\rib+v^{2\la_i+(\dt_i-2)\mu_i}\leb{\mu_i\atop\la_i}\rib.$$
It follows that
\begin{equation*}
\begin{split}
A(\dt,\la)&=v^{\la_i}(v^{\la_i+1}-v^{-\la_i-1})A(\dt-\bse_i,\la+\bse_i)
+v^{2\la_i}A(\dt-2\bse_i,\la)\\
&=-v^{-\la_i}(v^{\la_i+1}-v^{-\la_i-1})A(\dt+\bse_i,\la+\bse_i)
+v^{-2\la_i}A(\dt+2\bse_i,\la)
\end{split}
\end{equation*}
for $1\leq i\leq n$, $\la\in\mbnn$ and $\dt\in\mbzn$.
This shows that
$\Vn$ is spanned by $\fB_3$. Assume $$\sum_{A\in\Thnpm,\,\la,\dt\in\mbnn\atop\dt_i\in\{0,1\},\forall i}f_{A,\dt,\la}A(\dt,\la)=0$$
where $f_{A,\dt,\la}\in\mbq(v)$.  Then
$$\sum_{A\in\Thnpm\atop\mu\in\La(n,r-\sg(A))}\bigg
(\sum_{\la,\dt\in\mbnn\atop\dt_i\in\{0,1\},\forall i}f_{A,\dt,\la}
v^{\mu\centerdot\dt}\leb{\mu\atop\la}\rib\bigg)[A+\diag(\mu)]
=\sum_{A\in\Thnpm,\,\la,\dt\in\mbnn\atop\dt_i\in\{0,1\},\forall i}f_{A,\dt,\la}A(\dt,\la,r)=0.$$
This implies that $$\sum_{\la,\dt\in\mbnn\atop\dt_i\in\{0,1\},\forall i}f_{A,\dt,\la}
v^{\mu\centerdot\dt}\leb{\mu\atop\la}\rib=0$$ for $A\in\Thnpm$ and $\mu\in\mbnn$.
It follows that
$$\sum_{\la,\dt\in\mbnn\atop\dt_i\in\{0,1\},\forall i}f_{A,\dt,\la}0(\dt,\la,r)=
0$$
for $r\geq 0$ and $A\in\Thnpm$. Thus by \ref{basis of zero part} we conclude that
$f_{A,\dt,\la}=0$ for all $A,\dt,\la$. This shows that
the set $\fB_3$ is linearly independent and hence the set $\fB_3$ forms
a $\sZ$-basis for $\Vn$.
\end{proof}

We now use \ref{formula 1} and \ref{formula 2} to prove that $\Vn$ is a $\sZ$-subalgebra of $\prod_{r\geq 0}\bfSr$.

\begin{Prop}\label{Z algebra}
$\Vn$ is a $\sZ$-subalgebra of $\prod_{r\geq 0}\bfSr$. Furthermore
the elements $(mE_{h,h+1})(\bfl)$, $(mE_{h+1,h})(\bfl)$ and
$0(\dt,\la)$ $($for $m\in\mbn$, $1\leq h\leq n-1$, $\dt\in\mbzn$ and $\la\in\mbnn)$
generate $\Vn$ as a $\sZ$-algebra.
\end{Prop}
\begin{proof}
Let $\Vn_1$ be the $\sZ$-subalgebra of $\prod_{r\geq 0}\bfSr$ generated by
$(mE_{h,h+1})(\bfl)$, $(mE_{h+1,h})(\bfl)$ and
$0(\dt,\la)$ for $m\in\mbn$, $1\leq h\leq n-1$, $\dt\in\mbzn$ and $\la\in\mbnn$.
From \ref{formula 1} and \ref{formula 2} we see that
\begin{equation}\label{zeta(U(n))}
\Vn_1\han\Vn_1\Vn\han\Vn.
\end{equation}
So by \ref{basis for Vn} it is enough to prove $A(\bfl)0(\dt,\la)\in\Vn_1$ for $A\in\Thnpm$, $\dt,\la\in\mbnn$ with $\dt_i\in\{0,1\}$ ($1\leq i\leq n$).
We shall prove this by induction on $|\!|A|\!|$, where
$$|\!|A|\!|=\sum_{r<s}\frac{(s-r)(s-r+1)}{2}a_{rs}+\sum_{r>s}\frac{(r-s)(r-s+1)}{2}a_{rs}\in\mathbb N.$$
If $|\!|A|\!| = 0$, then $A(\bfl)0(\dt,\la)=0(\dt,\la)\in\Vn_1$.
Now we assume that $\ddet A>0$ and our statement is true
for $A'$ with $\ddet {A'}<\ddet A$.
According to \ref{tri}, for $A\in\Thnpm$, we have
$$\prod_{1\leq i\leq h<j\leq n}(a_{i,j}E_{h,h+1})(\bfl)
\cdot\prod_{1\leq j\leq h<i\leq
 n}(a_{i,j}E_{h+1,h})(\bfl)=A(\bfl)+f$$
where $f$ is the $\mbq(v)$-linear combination of $B(\bfl)0(\bfj)$ with $B\in\Thnpm$, $B\p A$ and $\bfj\in\mbzn$. It follows that
\begin{equation}\label{tri1}
\prod_{1\leq i\leq h<j\leq n}(a_{i,j}E_{h,h+1})(\bfl)
\cdot\prod_{1\leq j\leq h<i\leq
 n}(a_{i,j}E_{h+1,h})(\bfl)\cdot 0(\dt,\la)=A(\bfl)0(\dt,\la)+g
\end{equation}
for $\dt,\la\in\mbnn$ with $\dt_i\in\{0,1\}$ ($1\leq i\leq n$), where $g=f\cdot 0(\dt,\la)$. From \eqref{zeta(U(n))}, \ref{basis of zero part} and \ref{basis for Vn} we see that $g$ must be a $\sZ$-linear combination of $B(\bfl)0(\ga,\mu)$ with $B\in\Thnpm$, $B\p A$, $\ga,\mu\in\mbnn$ and $\ga_i\in\{0,1\}$ for $1\leq i\leq n$.
Note that if $B\in\Thnpm$ satisfy $B\p A$, then $|\!|B|\!| < |\!|A|\!|$ (see the proof of \cite[4.2]{BLM}). Thus by induction we conclude that $g\in\Vn_1$ and hence $A(\bfl)0(\dt,\la)\in\Vn_1$. The assertion follows.
\end{proof}

\begin{Thm}\label{realization}
There is a $\sZ$-algebra isomorphism
$\zeta:U(n)\ra\Vn$ satisfying
$E_h^{(m)}\map(mE_{h,h+1})(\bfl)$, $F_h^{(m)}\map(mE_{h+1,h})(\bfl)$ and
$\prod_{1\leq i\leq n}K_i^{\dt_i}\leb{K_i;0\atop \la_i}\rib\map 0(\dt,\la)$
for $m\in\mbn$, $1\leq h\leq n-1$, $\dt\in\mbzn$ and $\la\in\mbnn$.
\end{Thm}
\begin{proof}
The maps $\zeta_r:\bfU(n)\ra\bfSr$ induce an algebra homomorphism
\begin{equation*}
\zeta:\bfUn\ra\prod_{r\geq 0}\bfSr
\end{equation*}
 satisfying $\zeta(x)=(\zeta_r(x))_{r\geq 0}$ for $x\in\bfUn$.
From \ref{Z algebra} we see that $\zeta(U(n))=\Vn$. Furthermore by \ref{basis for Un}, \ref{basis for Vn} and \eqref{tri1}, we conclude that $\zeta$ is injective.
\end{proof}
\begin{Rem}
(1) Let $\bfVn$ be the $\mbq(v)$-subspace of $\prod_{r\geq 0}\bfSr$ spanned by $\{A(\dt)\mid A\in\Thnpm,\,\dt\in\mbzn\}.$ Then $\bfVn\cong\Vn\ot_\sZ\mbq(v)$. According to \ref{Z algebra} and \ref{realization} we conclude that $\bfVn$ is a $\mbq(v)$-subalgebra of $\prod_{r\geq 0}\bfSr$ and $\bfUn\cong\bfVn$.

(2) Note that for $A\in\Thnpm$ and $\la\in\La(n,r-\sg(A))$ we have $A(\bfl,\la,r)=[A+\diag(\la)]$. Thus from \ref{realization} we see that $\zeta_r(\Un)=\Sr$, which has been proved in \cite[3.4]{Du95}.
\end{Rem}

We now use $q$-Schur algebras over $k$ to realize quantum $\frak{gl}_n$ over $k$, where $k$ is a field containing
an $l$-th primitive root $\varepsilon$ of $1$ with $l\geq 1$ odd.
Specializing $v$ to $\varepsilon$, $k$ will be viewed as a $\sZ$-module. For $\mu\in\mbzn$ and $\la\in\mbnn$ we shall denote the image of $\bbl{\mu\atop\la}\bbr$ in $k$ by $\bbl{\mu\atop\la}\bbr_\vep$.
Let $\Unk=\Un\ot_\sZ k$ and $\Srk=\Sr\ot_\sZ k$.
By restriction, the map $\zeta_{r}:\bfUn\ra\bfSr$ induces
an algebra homomorphism $\zeta_r:\Un\ra\Sr$. By tensoring with the field $k$,
we get an algebra homomorphism $$\zeta_{r,k}:=\zeta_r\ot id:\Unk\ra\Srk.$$
Let $$\barUnk=\Unk/\lan K_i^l-1\mid 1\leq i\leq n-1\ran.$$
Since $\zeta_{r,k}(K_i^l)=1$, $\zeta_{r,k}$ induces an algebra homomorphism
$$\bar\zeta_{r,k}:\barUnk\ra\Srk$$
satisfying $\bar\zeta_{r,k}(\bar x)=\zeta_{r,k}(x)$ for $x\in\Unk$.
The maps $\bar\zeta_{r,k}$ induce an algebra homomorphism
$$\bar\zeta_k:=\prod_{r\geq 0}\bar\zeta_{r,k}:\barUnk\ra\prod_{r\geq 0}\Srk$$
satisfying $\bar\zeta_k(x)=(\bar\zeta_{r,k}(\bar x))_{r\geq 0}$ for $\bar x\in\barUnk$.
For $A\in\Thnr$ we let
$[A]_\vep=[A]\ot 1\in\Srk$. Similarly, for $A\in\Thnpm$, $\dt\in\mbzn$ and $\la\in\mbnn$,  let $A(\dt,\la,r)_\vep=A(\dt,\la,r)\ot 1\in\Srk$, $A(\dt,\la)_\vep=(A(\dt,\la,r)_\vep)_{r\geq 0}\in\prod_{r\geq 0}\Srk$
and $A(\dt)_\vep=A(\dt,\bfl)_\vep$. From \ref{basis for Vn} and \ref{realization}
we see that
$$\bar\zeta_k(\barUnk)=\spann_k\{A(\dt,\la)_\vep\mid A\in\Thnpm,\,\dt,\la\in\mbnn,\,\dt_i\in\{0,1\},\forall i\}.$$

\begin{Thm}\label{realization of barUnk}
The algebra homomorphism $\bar\zeta_k$ is injective. Furthermore, the
set $$\frak B_k:=\{A(\bfl)_\vep 0(-\la,\la)_\vep\mid
A\in\Thnpm,\,\la\in\mbnn\}$$ forms a $k$-basis of $\bar\zeta_k(\barUnk)$.
\end{Thm}
\begin{proof}
We will identify $\Un$ with $\Vn$ via the map $\zeta$ defined in \ref{realization}. From \ref{basis for Vn} and \cite[6.4(b)]{Lu90}, we see that
the set $$\{A(\bfl)0(l\dt)0(-\la,\la)\ot 1\mid A\in\Thnpm,\,\la,\dt\in\mbnn,\,\dt_i\in\{0,1\},\forall i\}$$
forms a $k$-basis for $\Unk$. It follows that $\bar\zeta_k(\barUnk)$ is spanned by the set $\frak B_k$. Thus it is enough to prove that the set $\frak B_k$ is linearly independent.

Assume $$\sum_{A\in\Thnpm,\,\la\in\mbnn}f_{A,\la}A(\bfl)_\vep0(-\la,\la)_\vep=0$$
where $f_{A,\la}\in k$. Then for any $r\geq 0$
\begin{equation*}
\sum_{A\in\Thnpm\atop\mu\in\La(n,r-\sg(A))}\bigg(\sum_{\la\in\mbnn}f_{A,\la}
\vep^{-\la\centerdot
(\mu+\co(A))}\leb{\mu+\co(A)\atop\la}\rib_\vep\bigg)[A+\diag(\mu)]_\vep
=0.
\end{equation*}
It follows that
for any $A\in\Thnpm$, $\mu\in\La(n,r-\sg(A))$ with $r\geq\sg(A)$, we have
\begin{equation}\label{sum}
\sum_{\la\in\mbnn}f_{A,\la}
\vep^{-\la\centerdot
(\mu+\co(A))}\leb{\mu+\co(A)\atop\la}\rib_\vep=0.
\end{equation}
We claim that for $A\in\Thnpm$ and $\mu,\al\in\mbnn$ we have
\begin{equation}\label{sum1}
\sum_{\la\in\mbnn,\,\la\geq\al}f_{A,\la}
\vep^{-\la\centerdot
(\mu+\co(A))}\leb{\mu+\co(A)\atop\la-\al}\rib_\vep=0.
\end{equation}
We apply induction on $\sg(\al)$. For $A\in\Thnpm$ and $\mu,\al\in\mbnn$, we denote
$$g_{A,\al,\mu}=\sum_{\la\in\mbnn,\,\la\geq\al}f_{A,\la}
\vep^{-\la\centerdot
(\mu+\co(A))}\leb{\mu+\co(A)\atop\la-\al}\rib_\vep.$$
 If $\sg(\al)=0$ then the claim follows from
\eqref{sum}. Now we assume $\sg(\al)>0$. There exist $\bt\in\mbnn$ such that
$\al=\bt+\bse_i$.
According to \ref{mul gauss binomial}(1), for $\la\in\mbnn$ with $\la\geq\bt$ we have
$$\leb{\mu+\bse_i+\co(A)\atop\la-\bt}\rib_\vep=\vep^{\la_i-\bt_i}\leb{\mu+\co(A)\atop
\la-\bt}\rib_\vep+
\vep^{\la_i-\bt_i-1-\mu_i-\sum_{1\leq k\leq n}a_{k,i}}\leb{\mu+\co(A)\atop\la-\bt-\bse_i}\rib_\vep.$$
Thus by the induction hypothesis we conclude that
\begin{equation*}
0=g_{A,\bt,\mu+\bse_i}=\vep^{-\bt_i}g_{A,\bt,\mu}+\vep^{-\bt_i-1-\sum_{1\leq k\leq n}a_{k,i}-\mu_i}g_{A,\al,\mu}=\vep^{-\bt_i-1-\sum_{1\leq k\leq n}a_{k,i}-\mu_i}g_{A,\al,\mu}.
\end{equation*}
for $A\in\Thnpm$ and $\mu\in\mbnn$.
It follows that $g_{A,\al,\mu}=0$ for $A\in\Thnpm$ and $\mu\in\mbnn$, proving \eqref{sum1}.

Let $\sX=\{\la\in\mbnn\mid f_{A,\la}\not=0\text{ for some }A\in\Thnpm\}$. If $\sX\not=\emptyset$, we may
choose a maximal element $\nu$ in $\sX$ with respect to $\leq$.
Then by \eqref{sum1} we have
\begin{equation*}
f_{A,\nu}=
\vep^{\nu\centerdot
(\mu+\co(A))}
\sum_{\la\in\mbnn,\,\la\geq\nu}f_{A,\la}
\vep^{-\la\centerdot
(\mu+\co(A))}\leb{\mu+\co(A)\atop\la-\nu}\rib_\vep=0.
\end{equation*}
for $A\in\Thnpm$. This is a
contradiction. Thus $f_{A,\la}=0$ for all $A\in\Thnpm$ and $\la\in\mbnn$.
The assertion follows.
\end{proof}
\begin{Rem}
(1) Let $\sU(\frak{gl}_n)$ be the universal enveloping algebra of $\frak{gl}_n$ and let $\sU_\mbz(\frak{gl}_n)$ be the Kostant $\mbz$-form of $\sU(\frak{gl}_n)$. Let $\Srq=\Sr\ot_\sZ\mbq$, $\Unz=\Un\ot_\sZ\mbz$, where $\mbz$ and $\mbq$ are regarded as $\sZ$-modules by specializing $v$ to $1$.
Let $\Wnz$ be the $\mbz$-submodule of $\prod_{r\geq 0}\Srq$ spanned by the set
$\{A(\bfl,\la)_1\mid A\in\afThnpm,\,\la\in\mbnn\}$.
According to \cite[6.7(c)]{Lu90}, \ref{realization} and \ref{realization of barUnk} we conclude that $\Wnz$ is a $\mbz$ algebra and $\sU_\mbz(\frak{gl}_n)\cong\Unz/\lan K_i-1\mid 1\leq i\leq n\ran\cong\Wnz.$

(2) Assume $\vep=1\in k$. Then $l=1$ and $\Srk$ is the Schur algebra over $k$.
Let $\Wnk$ be the $k$-subspace of $\prod_{r\geq 0}\Srk$ spanned by the set
$\{A(\bfl,\la)_1\mid A\in\afThnpm,\,\la\in\mbnn\}.$
From \cite[6.7(c)]{Lu90} and  \ref{realization of barUnk} we see that $\Wnk$ is a $k$-algebra and $\sU_\mbz(\frak{gl}_n)\ot_\mbz k\cong\barUnk\cong\Wnk.$
\end{Rem}

We end this paper with a conjecture on affine $q$-Schur algebras.
Let
 $\afThn$  be the set of all $\mbz\times\mbz$ matrices
$A=(a_{i,j})_{i,j\in\mbz}$ with $a_{i,j}\in\mbn$ such that
\begin{itemize}
\item[(a)]$a_{i,j}=a_{i+n,j+n}$ for $i,j\in\mbz$, and \item[(b)] for
every $i\in\mbz$, both sets $\{j\in\mbz\mid a_{i,j}\not=0\}$  and
$\{j\in\mbz\mid a_{j,i} \not= 0\}$ are finite.
\end{itemize}
Let
$\afmbzn=\{(\la_i)_{i\in\mbz}\mid
\la_i\in\mbz,\,\la_i=\la_{i-n}\ \text{for}\ i\in\mbz\}$ and $\afmbnn=\{(\la_i)_{i\in\mbz}\in \afmbzn\mid \la_i\ge0\}$.
For $r\in\mbn$ let $\afThnr=\{A\in\afThn\mid\sg(A)=r\}$ and
$\afLanr=\{\la\in\afmbnn\mid\sg(\la)=r\}$
where
$\sg(\la)=\sum_{1\leq i\leq n}\la_i$ and $\sg(A)=\sum_{1\leq i\leq n,\,
j\in\mbz}a_{i,j}$. For $\la\in\afLanr$, let
$\diag(\la)=(\dt_{i,j}\la_i)_{i,j\in\mbz}\in\afThnr$.

Let $\afSr$ be the affine $q$-Schur algebra over $\sZ$. It has a normalized $\sZ$-basis
$\{[A]\mid A\in\afThnr\}$ (see \cite[1.9]{Lu99}). We put $\afbfSr=\afSr\otimes_\sZ\mathbb Q(v).$

Let
$\afThnpm=\{A\in\afThn\mid a_{i,i}
=0\text{ for all $i$}\}.$
For $A\in\afThnpm$, $\dt\in\afmbzn$ and $\la\in\afmbnn$
let
\begin{equation*}
\begin{split}
A(\dt,\la,r)&=\sum_{\mu\in\afLa(n,r-\sg(A))}v^{\mu\centerdot\dt}
\leb{\mu\atop\la}\rib[A+\diag(\mu)]\in\afbfSr\\
A(\dt,\la)&=(A(\dt,\la,r))_{r\geq 0}\in\prod_{r\geq 0}\afbfSr
\end{split}
\end{equation*}
where $\mu\centerdot\dt=\sum_{1\leq i\leq n}\mu_i\dt_i$ and $\leb{\mu\atop\la}\rib=\big[{\mu_1\atop\la_1}\big]\cdots\big[{\mu_n\atop\la_n}\big]$.
Let $A(\dt)=A(\dt,\bfl)$, where
$\bfl=(\cdots,0,\cdots,0,\cdots)\in\afmbnn$.

We shall denote by $\afVn$  the $\sZ$-submodule of $\prod_{r\geq 0}\afbfSr$ spanned by
$\{A(\dt,\la)\mid A\in\afThnpm,\,\dt\in\afmbzn,\,\la\in\afmbnn\}.$

\begin{Lem}\label{basis for afVn}
Each of the following set forms a $\sZ$-basis for $\afVn:$

$(1)$ $\{0(\dt,\la)A(\bfl)\mid A\in\afThnpm,\,
\dt,\la\in\afmbnn,\,\dt_i\in\{0,1\},\forall i\};$

$(2)$ $\{A(\bfl)0(\dt,\la)\mid A\in\afThnpm,\,
\dt,\la\in\afmbnn,\,\dt_i\in\{0,1\},\forall i\};$

$(3)$ $\{A(\dt,\la)\mid A\in\afThnpm,\,
\dt,\la\in\afmbnn,\,\dt_i\in\{0,1\},\forall i\}.$
\end{Lem}
\begin{proof}
The assertion can be proved in a way similar to the proof of \ref{basis for Vn}.
\end{proof}

According to \ref{Z algebra}, $\Vn$ is a $\sZ$-subalgebra of $\prod_{r\geq 0}\bfSr$. Thus, it is natural to formulate the following conjecture.
\begin{Conj}\label{conj}
$\afVn$ is a $\sZ$-subalgebra of $\prod_{r\geq 0}\afbfSr$.
\end{Conj}
\begin{Rems}
(1) According to \cite{Fu}, Conjecture \ref{conj} is true in the classical ($v = 1$) case.

(2)
Let $\afbfVn$ be the $\mbq(v)$-subspace of $\prod_{r\geq 0}\afbfSr$ spanned by all
$A(\dt)$ for $A\in\afThnpm$ and $\dt\in\mbzn$.
It is conjectured in \cite[5.5(2)]{DF09} that $\afbfVn$ is a $\mbq(v)$-subalgebra of $\prod_{r\geq 0}\afbfSr$.
From \ref{basis of zero part} and \ref{basis for afVn}, we see that $\afbfVn\cong\afVn\ot\mbq(v)$. Thus if Conjecture \ref{conj} is true, then we conclude that
$\afbfVn$ is a $\mbq(v)$-subalgebra of $\prod_{r\geq 0}\afbfSr$.

(3) If Conjecture \ref{conj} is true, then by \cite[3.7.3]{DDF} we conclude that the conjecture formulated in \cite[3.8.6]{DDF} is true and $\afVn$ is isomorphic to $\ti{\fD}_\vtg(n)$, where $\ti{\fD}_\vtg(n)$ is a certain $\sZ$-module defined in \cite[(3.8.1.1)]{DDF}.
\end{Rems}

\end{document}